\documentclass[a4paper,12pt]{amsart}

\usepackage[latin1]{inputenc}
\usepackage{tabmac}
\usepackage{url}
\usepackage{amsmath, amssymb}
\usepackage{braket}
\usepackage[vcentermath,enableskew]{youngtab}

\newcommand{\Fillingson}[1]{\mathcal{F}_{#1}}
\newcommand{\Fillingsonwith}[2]{\mathcal{F}_{#1,#2}}

\newcommand{\numof}[1]{\left| #1 \right|}

\newcommand{\Des}{\operatorname{Des}}
\newcommand{\Inv}{\operatorname{Inv}}
\newcommand{\maj}{\operatorname{maj}}
\newcommand{\inv}{\operatorname{inv}}
\newcommand{\leg}{\operatorname{leg}}
\newcommand{\arm}{\operatorname{arm}}
\newcommand{\Att}{\operatorname{Att}}

\newcommand{\denotes}{:=}

\newtheorem{thm}{Theorem}[section]
\newtheorem{lemma}[thm]{Lemma}

\newtheorem{theorem}[thm]{Theorem}
\newtheorem{proposition}[thm]{Proposition}
\newtheorem{algo}[thm]{Algorithm}
\newtheorem{remark}[thm]{Remark}

\newtheorem{example}[thm]{Example}

\newtheorem{definition}[thm]{Definition}

\title[Factorization formulas for Macdonald polynomials]{A bijective proof of a factorization formula for Macdonald polynomials at roots of unity}
\author[F. Descouens and H. Morita and Y. Numata]{F. Descouens, H. Morita and Y. Numata}

\address[F. Descouens]{Fields Institute, 222 College Street, Toronto, Ontario, M5T 3J1, Canada}
\email{fdescoue@fields.utoronto.ca}
\address[H. Morita]{Oyama National College of Technology, Nakakuki 771, Oyama, 323-0806, Japan}
\email{morita@oyama-ct.ac.jp}
\address[Y. Numata]{Department of Mathematics, Hokkaido University, Kita 10, Nishi 8, Kita-Ku, Sapporo, Hokkaido, 060-0810, Japan}
\email{nu@math.sci.hokudai.ac.jp}

\begin{document}
\maketitle
\begin{abstract}
We give a combinatorial proof of the factorization formula of modified
 Macdonald polynomials $\widetilde{H}_\lambda(X;q,t)$ when $t$ is
 specialized at a primitive root of unity. Our proof is restricted to
 the special case where $\lambda$ is a two columns partition. We mainly
 use the combinatorial interpretation of Haiman, Haglund and Loehr
 giving the expansion of $\widetilde{H}_\lambda(X;q,t)$ on the monomial
 basis. 
%
\end{abstract}
\section{Introduction}
The different versions of Macdonald polynomials have been intensively studied
from a combinatorial and algebraic approach since their introduction in
\cite{Macdo1}. These polynomials are deformations with two parameters of
usual symmetric functions and generalize the Hall-Littlewood
functions. We are mainly interested in the modified version of Macdonald
polynomials $\widetilde{H}_\lambda(X;q,t)$. In \cite{HHL}, Haglund,
Haiman and Loehr give a combinatorial interpretation of the expansion of
these modified Macdonald polynomials in the monomial basis. This
combinatorial interpretation is based on the definition of two
statistics $inv(T)$ and $maj(T)$ on the set $\mathcal{F}_{\mu, \nu}$ of
all the fillings $T$ of a given shape $\mu$ and evaluation $\nu$. Hence,
we have the following formula 
\begin{equation*}
\widetilde{H}_{\mu}(X;q,t) = \sum_{\nu} \left (\sum_{T\in \mathcal{F}_{\mu,\nu}} q^{\inv(T)} t^{\maj(T)} \right ) X^T \ .
\end{equation*} 
 In \cite{DM}, the authors give an algebraic proof of factorization formulas for these polynomials, when the parameter $t$ is specialized at primitive roots of unity. More precisely, for any positive integer $n$ and any partition $\mu$ such that $\mu=(\mu',n^l,\mu'')$, we have  
\begin{gather}\label{factorization}
\widetilde{H}_{\mu}(X;q,\zeta_l)
=\widetilde{H}_{(\mu',\mu'')}(X;q,\zeta_l)\cdot\widetilde{H}_{(n^l)}(X;q,\zeta_l)\ ,
\end{gather}
where $\zeta_l$ is an $l$-th primitive root of unity.
We propose to give a combinatorial proof of this formula in the special case where $\mu''=\varnothing$ and $n=1$ or $2$. 

\section{Combinatorial interpretation for Macdonald polynomials}\label{sec:def}
We mainly follow the notations of \cite{Macdo2} for symmetric functions. We recall the combinatorial interpretation of the expansion of modified Macdonald polynomials on the monomials basis given in \cite{HHL}.

A partition $\lambda$ is a sequence of positive integers
$(\lambda_1,\ldots, \lambda_n)$ such that $\lambda_1 \ge \ldots \ge
\lambda_n$. We represent such a partition by its Young diagram using the
French convention. For a given cell $u$ of $\lambda$, the arm of $u$,
denoted by $\arm(u)$, is the number of cells strictly to the right of
$u$. The leg of $u$, denoted by $\leg(u)$, is the number of cells strictly
above $u$.  
\begin{example} The partition $(4,3,2)$ can be represented by the following diagram
$$
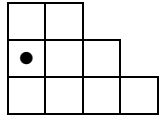
\tableau[sbY]{ & \\\bullet & & \\ & & &} \quad
$$
For the cell $\bullet$, we have $\arm(\bullet)=2$ and $\leg(\bullet)=1$.
\end{example}

We call $T$ a filling of shape $\lambda$ 
if  $T$ is a tableau obtained by
assigning integer entries to the cells of the diagram of $\lambda$ with
no increasing conditions. 
The evaluation of a filling $T$ is the vector
where the $i$-th entry is the number of cells labeled by $i$ in $T$. 
The set of all the fillings of
shape $\lambda$ and evaluation $\mu$ is denoted by
$\mathcal{F}_{\lambda,\mu}$. 

A descent of a filling $T$ is a pair of cells satisfying the following
condition  
\begin{equation*}
T_{i+1,j}\ >\ T_{i,j} \ . 
\end{equation*}
For a given filling $T$, we define the set $\Des(T)$ of the descents of
$T$ by 
\begin{equation*}
\Des(T)  = \{ T_{i+1,j}\ \text{ such that }\ T_{i+1,j} > T_{i,j} \}\ .
\end{equation*}
 The statistic $\maj(T)$ is defined by
\begin{equation*}
\maj(T) = \sum_{u \in \Des(T)} (\leg(u)+1) \ . 
\end{equation*}

\begin{example}\label{filling} 
The following tableau is
a filling of shape $(4,3,2)$ and evaluation $(1,2,1,3,0,1,0,1)$:
$$
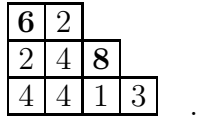
\tableau[sbY]{\mathbf{6}  & 2 \\ 2 & 4 & \mathbf{8} \\ 4 & 4 & 1 & 3}\quad .
$$
The descent set of this filling is $\Des(T)=\Set{(3,1),\ (2,3)}$. Hence,
\begin{equation*}
\maj(T) = 2\ .
\end{equation*}
\end{example} 
Two cells of a filling are said to attack each other if either 
\begin{enumerate}
\item they are in the same row, or 
\item they are in consecutive rows, 
      with the cell in the upper row strictly to the right of the one in
      the lower row.  
\end{enumerate}
\begin{example}
The following picture shows the two kinds of attacking cells:
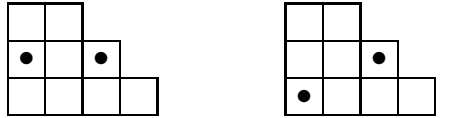
\begin{equation*} 
\tableau[sbY]{ & \\ \bullet & & \bullet\\ & & &} \quad \quad \quad \quad 
\tableau[sbY]{ & \\ & & \bullet \\\bullet & & &  }\quad .
\end{equation*}
\end{example}
The reading order of a filling is the row by row reading from top to
bottom and left to right within each row. 
A pair $(u,v)$ of cells is an inversion 
if they satisfy the three following conditions: 
\begin{enumerate}
\item they are attacking each others,
\item $T_{u} < T_{v}$, and
\item the cell $v$ appears before the cell $u$ in the
      reading order. 
\end{enumerate}  
The number of inversions of $T$ is denoted by $\Inv(T)$. 
The statistic $\inv(T)$ is defined by 
\begin{equation*}
\inv(T) = \Inv(T) - \sum_{u\in \Des(T)} \arm(u)\ .
\end{equation*}
\begin{example} 
For the filling $T$ of Example \ref{filling}, we have
 $\inv(T)=8-1=7$.  
\end{example}
\begin{theorem}[\cite{HHL}]
The modified Macdonald polynomial $\widetilde{H}_\mu(x;q,t)$ 
has a description as the following weighted generating function over fillings of shape $\mu$
\begin{gather*}
 \widetilde{H}_\mu(X;q,t)  \denotes \sum_{\nu}\sum_{T\in \mathcal{F_{\mu,\nu}}}
q^{\inv(T)} t^{\maj(T)} X^T\ ,
\end{gather*}
where the sum is over all the compositions $\nu$ of size $\vert \mu \vert.$
\end{theorem}
\begin{theorem}[\cite{DM}]\label{ThDM}
For any positive integer $n$ and any partition $\mu$ such that $\mu=(\mu',n^l,\mu'')$, we have  
\begin{gather*}
\widetilde{H}_{\mu}(X;q,\zeta_l)
=\widetilde{H}_{(\mu',\mu'')}(X;q,\zeta_l)\cdot\widetilde{H}_{(n^l)}(X;q,\zeta_l)\ ,
\end{gather*}
where $\zeta_l$ is an $l$-th primitive root of unity.
\end{theorem}
We now give some technical definitions on fillings,
which are needed later in our proof. 
\begin{definition}
For any partition $\mu$, we define 
the sets
$\Att_{i}(\mu)$ and
$\Att_{i,i-1}(\mu)$ 
of pairs of boxes of $\mu$ by
\begin{align*}
\Att_{i}(\mu)   & \denotes \Set{((i,j),(i,k))|1\leq j < k \leq \mu_i}\ , \\
\Att_{i,i-1}(\mu)   & \denotes \Set{((i,j),(i-1,k))|1\leq j < k \leq \mu_i}.
\end{align*}
The union of these two sets gives us the attacking cells of $\mu$ coming from its $i$-th row.
\end{definition}
\begin{definition}
For a filling $T$ of shape $\mu$, 
we define the set $\Des_{i,i-1}(T)$ of pairs of boxes of $\mu$ by
\begin{align*}
\Des_{i,i-1}(T) & \denotes \Set{(i,j)\in\mu | T_{i,j} > T_{i-1,j}}\ .
\end{align*}
This set is the restriction of the descents set $\Des(T)$ to the
 descents which occurs in the $i$-th row of $\mu$. 
Let us now define the
 following restrictions of the quantities $\sum_{u\in \Des(T)}
 \arm(u)$ and $\sum_{u\in \Des(T)} \leg(u) $ 
\begin{align*}
\arm_{i,i-1}(T) & \denotes  \sum_{b\in \Des_{i,i-1}(T)} \arm(b)\ ,\\
\maj_{i,i-1}(T) & \denotes  \sum_{b\in \Des_{i,i-1}(T)} (1+\leg(b))\ .
\end{align*}
\end{definition}
\begin{example}\label{example:notation}
For the following filling $T$  
\begin{equation*}
\young(1,47,32,56)\ ,
\end{equation*}
the set $\Des_{3,2}(T)=\Set{(3,1),(3,2)}$ consists of
the boxes where $4$ and $7$ lie. The sets $\Des_{2,1}(T)$ and $\Des_{4,3}(T)$ are reduced to the empty set. Hence we have
\begin{align*}
\arm_{3,2}(T) &= 1+0=1,&
\arm_{2,1}(T) &= \arm_{4,3}(T)=0,\\
\maj_{3,2}(T) &= 2+1=3, &
\maj_{2,1}(T) &= \maj_{4,3}(T)=0.
\end{align*}
\end{example}
\begin{definition}
We define the subset $\Inv_{i}(T)$ (resp. $\Inv_{i,i-1}(T)$) of
 $\Att_{i}(T)$ (resp. $\Att_{i,i-1}(T)$) by
\begin{align*}
\Inv_{i}(T)     & \denotes  \Set{(b,c)\in \Att_{i}(\mu)|T_b > T_c}\ , \\
\Inv_{i,i-1}(T) & \denotes  \Set{(b,c)\in \Att_{i,i-1}(\mu)|T_b > T_c}\ .
\end{align*}
The union of these sets gives us the inversions of $T$ which are coming from the $i$-th row of $\mu$.
 Let now define the corresponding restriction of the statistic $\inv(T)$ by
\begin{align*}
\inv_{i,i-1}(T)     & \denotes \numof{\Inv_{i}(T)}+\numof{\Inv_{i,i-1}(T)} -\arm_{i,i-1}(T).
\end{align*}
\end{definition}
\begin{example}\label{example:natation2}
For the filling $T$ of Example \ref{example:notation}, we have
\begin{align*}
\Inv_{2}(T)&=\Set{ (2,1),(2,2) }, &
\Inv_{1}(T)&=\Inv_{3}(T)=\Inv_{4}(T)=\emptyset, \\
\Inv_{3,2}(T)&=\Set{ (2,1),(3,2)}, &
\Inv_{2,1}(T)&=\Inv_{4,3}(T)=\emptyset\ .
\end{align*}
Hence we have 
\begin{align*}
\inv_{2,1}(T) & =  1+0-0=1,\\ 
\inv_{3,2}(T) & =  0+1-1=0,\\
\inv_{4,3}(T) & =  0+0-0=0.
\end{align*}
\end{example}
Using all these restrictions, we can express the statistics $\maj(T)$ and $\inv(T)$  by
\begin{align*}
\maj(T) \denotes \sum_{i=2}^{l(\mu)} \maj_{i,i-1}(T) 
\quad \text{and} \quad 
\inv(T) \denotes \numof{\Inv_{1}(T)} + \sum_{i=2}^{l(\mu)} \inv_{i,i-1}(T).
\end{align*}
\begin{example}
Let $T$ be the filling of Example $\ref{example:notation}$. From computations of Examples
 $\ref{example:notation}$ and  $\ref{example:natation2}$, we obtain the following statistics
\begin{align*}
\maj(T)= 0+3+0=3
\quad \text{and} \quad
\inv(T)= 0+(1+0+0)=1.
\end{align*}
\end{example}
\section{Main results}
For two compositions
$\nu'=(\nu'_1,\ldots,\nu'_k)$ and
$\nu''=(\nu''_1,\ldots,\nu''_k)$, 
$\nu'+\nu''$ denotes 
the composition $(\nu'_1+\nu''_1,\ldots,\nu'_k+\nu''_k)$.
Let $\mu$ be a partition such that $\mu=(\mu^\prime, n^l, \mu'')$ such
that $\mu'_{l(\mu')}\geq n$ and $\mu''_{1}\leq n$. 
In order to prove combinatorially Theorem \ref{ThDM}, 
we have to define two bijections $\tau$ and $\pi^{\ast}$ between
different sets of fillings   
\begin{align*}
\left \{
\begin{array}{cccl}
\tau\colon      & \mathcal{F}_{\mu,\nu} & \longrightarrow & \mathcal{F}_{\mu,\nu}\ ,\\
\pi^{\ast}\colon& \mathcal{F}_{\mu,\nu} & \longrightarrow & \bigcup_{\nu=\nu'+\nu''}
\mathcal{F}_{(\mu',\mu''),\nu'} \times \mathcal{F}_{(n^l),\nu''}\ ,
\end{array}\right .
\end{align*}
with
\begin{align*}
\left \lbrace
\begin{array}{ccl}
\maj(\tau(T)) & \equiv &  \maj(\pi^{\ast}(T)) \pmod{l}\ ,\\
\inv(\tau(T)) &     =  & \inv(\pi^{\ast}(T)) \ .
\end{array}\right .
\end{align*}
The definition of the statistics $maj$ and $inv$ are extended on couples of fillings by 
\begin{align*}
\left \lbrace
\begin{array}{ccl}
\maj(\pi^{\ast}(T)) & := &  \maj(\pi^{\ast}(T)_{1}) +\maj(\pi^{\ast}(T)_{2})\ , \\
\inv(\pi^{\ast}(T)) & := & \inv(\pi^{\ast}(T)_{1}) +\inv(\pi^{\ast}(T)_{2})\ .
\end{array}\right .
\end{align*}
We restrict ourselves to the case $n=1$ or $2$ and Young diagrams $\mu$
with tails,  
i.e.,
$$\mu=(\mu', n^l) \quad \text{and} \quad \mu^\prime_{l(\mu')} \geq n\ .$$
Hence, the factorization formula (\ref{factorization}) becomes  
\begin{equation*}
\widetilde{H}_{(\mu',n^l)}(X;q,\zeta_l)=\widetilde{H}_{\mu'}(X;q,\zeta_l)\cdot\widetilde{H}_{(n^l)}(X;q,\zeta_l)\ .
\end{equation*}
For the factorization in the case when $n=1$ or $2$, 
we give a bijective proofs in Theorems \ref{thm:n=1} and \ref{thm:n=2},
and the proofs are detailed in Section \ref{sec:proofn=2}.

Let 
$\pi\colon \mu' \cup (n^l) \to \mu=(\mu',n^l)$  
be the natural bijection, i.e., 
\begin{align*}
\begin{cases}
\pi(i,j)  =  (i,j) &\text{if $(i,j)\in \mu'$,} \\
\pi(i,j)  =  (i+l(\mu'),j) &\text{if $(i,j) \in (n^l)$.}
\end{cases}
\end{align*}
The map $\pi$ on partitions induces the following bijection on fillings
\begin{equation}\label{def:pistar}
\pi^{\ast}\colon \mathcal{F}_{\mu,\nu} 
\longrightarrow
\bigcup_{\nu=\nu'+\nu''} \mathcal{F}_{\mu',\nu'} \times \mathcal{F}_{(n^l),\nu''}\ ,
\end{equation}
defined for all $T$ in $\mathcal{F}_{\mu,\nu}$ by
$$
\begin{cases}
(\pi^{\ast}(T))_1  =  (T_{i,j})_{(i,j)\in \mu'}\ , \\
(\pi^{\ast}(T))_2  =  (T_{i+l(\mu'),j})_{(i,j) \in (n^l)}\ .
\end{cases}
$$
\begin{proposition}
For a filling $T$ of shape $\mu$,
let $(T', T'')$ be an element of $\Fillingson{\mu',\nu'} \times \Fillingson{(n^l),\nu''}$
satisfying the condition  $\pi^{\ast}(T) = (T', T'')$.
Then 
\begin{align*}
\pi^{\ast-1}(\Des_{i+1,i}(T))&=\Des_{i+1,i}(T'),     &
\pi^{\ast-1}(\Des_{k+i+1,k+i}(T))&=\Des_{i+1,i}(T''),\\  
\pi^{\ast-1}(\Inv_{i+1,i}(T))&=\Inv_{i+1,i}(T'),     &
\pi^{\ast-1}(\Inv_{k+i+1,k+i}(T))&=\Inv_{i+1,i}(T''),\\  
\pi^{\ast-1}(\Inv_{i}(T))&=\Inv_{i}(T'),         &
\pi^{\ast-1}(\Inv_{k+i}(T))&=\Inv_{i}(T'')\ .
\end{align*}
Hence we have the following equations
\begin{align*}
\maj_{i,i-1}(T)&=\maj_{i,i-1}(T') + l\cdot \numof{\Des_{i,i-1}(T')},&
\maj_{k+i,k+i-1}(T)&=\maj_{i,i-1}(T''), \\
\inv_{i,i-1}(T)&=\inv_{i,i-1}(T') ,     &
\inv_{k+i,k+i-1}(T)&=\inv_{i,i-1}(T'')\ .
\end{align*}
This implies the following expression for $\maj(T)$ and $\inv(T)$
\begin{equation*}
\left \{
\begin{array}{ccc}
\maj(T)
&\equiv & \maj(T')+\maj(T'')+\maj_{l(\mu')+1,l(\mu')}(T) \pmod{l}\ ,\\
\inv(T) &= & \inv(T')+\inv(T'')+\numof{\Inv_{l(\mu')+1,l(\mu')}}-\arm_{l(\mu')+1,l(\mu')}(T)\ .
\end{array}\right .
\end{equation*}
\end{proposition}
\subsection{The case n=1}
Let $\mu$ be a partition of the form $\mu=(\mu'_1,\ldots,\mu'_k,1^l)$. In this special case, we have 
\begin{equation*}
\Att_{k+1,k}=\emptyset \quad \text{and} \quad \Inv_{k+1,k}(T)=\emptyset \ .
\end{equation*}
The cell $b=(k+1,1)\in \mu$ is the unique candidate for being an element of $\Des_{k+1,k}(T)$. Hence
\begin{equation*}
\arm(b)=0\ .
\end{equation*}
Consequently $\arm_{k+1,k}(T)=0$ and $\inv(T)=\inv(\pi(T))$.

Since $\maj_{k+1,k}(T)=l\numof{\Des_{k+1,k}(T)}$ and
$\maj_{k+1,k}(T)\equiv 0 \pmod{l}$,
we have 
\begin{equation*}
\maj(T) \equiv \maj(\pi(T)) \pmod{l}\ .
\end{equation*}
Hence, we can take the identity map $id$ for $\tau$ in order to obtain a combinatorial proof of Theorem \ref{ThDM} in the case $n=1$.

\begin{theorem}\label{thm:n=1}
For a partition $\mu=(\mu'_1,\ldots,\mu'_k, 1^l)$,
let 
$\pi\colon \mu' \cup (1^l) \to \mu$ be the natural bijection and 
$\pi^{\ast}\colon \Fillingsonwith{\mu}{\nu} \to \bigcup\Fillingsonwith{\mu'}{\nu'}\times \Fillingsonwith{(1^l)}{\nu''}$ 
be the bijection induced by $\pi$ as defined in (\ref{def:pistar}). Let
$\tau$ be the identity map on $\Fillingsonwith{\mu}{\nu}$.
Then $\pi^{\ast}$ and $\tau$ satisfy
\begin{align*}
\left \lbrace
\begin{array}{ccl}
\maj(\tau(T)) & \equiv & \maj(\pi^{\ast}(T)) \pmod{l}\ ,\\
\inv(\tau(T)) &     =  & \inv(\pi^{\ast}(T)) \ .
\end{array}\right .
\end{align*}
\end{theorem}

\begin{example}
Let us consider the case $l=3$ and $\mu=(2,2,1,1,1)$.
In this case, we have
\begin{align*}
\maj\left (\ \young(2,1,3,23,12)\ \right) &=1+3+4+1=9 ,&\maj\left( \
 \young(23,12)\ \right ) + \maj\ \left(\ \young(2,1,3)\ \right)&=(1+1)+1=3,\\ \\
\inv \left (\ \young(2,1,3,23,12)\ \right ) &=1-1 =0,
&\inv\left ( \ \young(23,12)\ \right) + \inv\left (\ \young(2,1,3)\ \right )&=(1-1)+0=0.
\end{align*}
\end{example}
\subsection{The case $n=2$}
First we determine two conditions in order to define the appropriate
$\tau$. 
We first define some technical conditions on fillings which we
will permit us to define the elementary steps of Algorithm \ref{def:tau}.  

\begin{definition}[Condition $xAx$]
A filling $\ \young(ab,A\ )\ $ satisfies
the condition $xAx$ if 
one of the following conditions holds
\begin{equation*}
a\leq A < b \quad \text{or}\quad b\leq A < a\ .
\end{equation*}
\end{definition}
\begin{definition}[Condition $xXxX$]
A filling $\ \young(ab,AB)\ $ satisfies
the condition $xXxX$ if 
one of the following conditions holds
\begin{equation*}
\begin{array}{ccccc}
a\leq A < b \leq B\ ,&\quad& A < b \leq B <a\ ,\\
b\leq A < a \leq B\ ,&\quad& A < a \leq B <b\ ,\\
a\leq B < b \leq A\ ,&\quad& B < b \leq A <a\ ,\\
b\leq B < a \leq A\ ,&\quad& B < a \leq A <b\ . 
\end{array}
\end{equation*}
\end{definition}
\begin{proposition}\label{remark:symmeticityII} 
We have the following property on the conditions $xAx$ and $xXxX$
\begin{enumerate}
\item If a filling $\ \young(ab,A\ )\ $ satisfies the condition $xAx$, 
then $\ \young(ba,A\ )\ $ also satisfies the condition $xAx$,
\item If a filling $\ \young(ab,AB)\ $ satisfies the condition $xXxX$, 
then $\ \young(ba,BA)\ $ also satisfies the condition $xXxX$.
\end{enumerate}
\end{proposition}
We give an algorithm which permits to determine $\tau$ for any filling
of shape $\mu=(\mu'_1,\ldots, \mu'_k,2^l)$ with $\mu'_k \geq 2$.
\begin{algo}[Definition of $\tau$]\label{def:tau} 
\mbox{}
 \begin{itemize}
  \item {\bf Input:} A filling $T$ and $k$.
  \item {\bf Procedure} 
	\begin{enumerate}
	 \item[$\rhd$] Initialization of variables
		       \begin{enumerate}
			\item  $i\longleftarrow k\ $,
			\item  $\ T' \longleftarrow T$.
		       \end{enumerate}
	 \item[$\rhd$] If the $i$-th row and the $(i+1)$-th row of $T'$ satisfy
		       the condition $xAx$  do 
		       \begin{enumerate}
			\item  swap the two values in the $(i+1)$-th row of
			       $T'$,
			\item  $i \longleftarrow i+1$.
		       \end{enumerate}
		       else return $T'$.
	 \item[$\rhd$] While the $i$-th row and the $(i+1)$-th row of $T'$
		       satisfy the condition $xXxX$ do 
		       \begin{enumerate}
			\item  swap the two values in $(i+1)$-th row of
			       $T'$,
			\item  $i \longleftarrow i+1$.
		       \end{enumerate}
	\end{enumerate}
  \item {\bf Output:} The filling $T^\prime$.
 \end{itemize}
\end{algo}
\begin{example}\label{example:tau}
For $l=5$ and the following filling $T$, the steps of the algorithm are  
\begin{equation*}
T\ =\ \tableau[sbY]{1&4\\3&5\\2&6\\1&3\\\boxed{\mathbf{2}}&\boxed{\mathbf{4}}\\\boxed{\mathbf{3}}&3&3\\4&4&4}
\quad \longrightarrow\quad \tableau[sbY]{1 &4\\3&5\\2&6\\\boxed{\mathbf{1}} & \boxed{\mathbf{3}}\\\boxed{\mathbf{4}}&\boxed{\mathbf{2}}\\3&3&3\\4&4&4} 
\quad \longrightarrow\quad \tableau[sbY]{1&4\\3&5\\\boxed{\mathbf{2}}&\boxed{\mathbf{6}}\\\boxed{\mathbf{3}}&\boxed{\mathbf{1}}\\4&2\\3&3&3\\4&4&4} 
\quad \longrightarrow\quad \tableau[sbY]{1&4\\3&5\\6&2\\3&1\\4&2\\3&3&3\\4&4&4}\ =\ \tau(T)\ .
\end{equation*}
We have put in bold the cells which occur at each step of the
 algorithm. 
The first step corresponds to the condition $xAx$ and 
the others to the condition $xXxX$. 
\end{example}
\begin{proposition}
The application $\tau$ determined by the Algorithm \ref{def:tau} is an
 involution and a bijection. 
\end{proposition}
\begin{proof}
The fact that $\tau$ is an involution follows directly from Proposition
 \ref{remark:symmeticityII}. 
Moreover, as each step of the Algorithm \ref{def:tau} is invertible, 
the map $\tau$ is a bijection on $\mathcal{F}_{\mu,\nu}$. 
\end{proof}
\begin{theorem}\label{thm:n=2}
For a partition $\mu=(\mu'_1,\ldots,\mu'_k, 2^l)$ such that 
$\mu'_k\geq 2$, let 
 $\pi^{\ast}$ be the natural bijection defined in (\ref{def:pistar}),
 and $\tau$ be the involution determined by Algorithm \ref{def:tau}.
Then $\pi^{\ast}$ and $\tau$ satisfy
\begin{align*}
\left \lbrace
\begin{array}{ccl}
\maj(\tau(T)) & \equiv & \maj(\pi^{\ast}(T)) \pmod{l}\ ,\\
\inv(\tau(T)) &     =  & \inv(\pi^{\ast}(T)) \ .
\end{array}\right .
\end{align*}
\end{theorem}
\begin{example}
Let $l=5$ and $T$ be the the filling of Example \ref{example:tau}.
For the statistic $\maj$, we have
\begin{align*}
\left \{
\begin{array}{ccl}
\maj(\tau(T))= &\maj\left(\ \young(14,35,62,31,42,333,444)\ \right)
& =13\ , \\
\text{and}\\
\maj(\pi^{\ast}(T))=&\maj\left (\ \young(333,444)\ \right)+\maj\left(\ \young(14,35,62,31,42)\ \right)
&=0+8=8\equiv 13 \mod 5\ .
\end{array}\right .
\end{align*}
And for the statistic $\inv$, we have
\begin{align*}
\left \{
\begin{array}{ccl}
\maj(\tau(T))= & \inv\left(\ \young(14,35,62,31,42,333,444)\ \right) 
& =2\ , \\
\maj(\pi^{\ast}(T))= & \inv\left(\ \young(333,444)\ \right)+\inv\left(\ \young(14,35,62,31,42)\ \right)
&=0+2=2 \ .
\end{array} \right .
\end{align*}
\end{example}
\section{Proof of the main Theorem}\label{sec:proofn=2}
In order to prove Theorem \ref{thm:n=2}, i.e 
\begin{align*}
\maj(\pi^{\ast}(T))\equiv \maj(\tau(T)) \pmod{l} \quad \text{and} \quad 
\inv(\pi^{\ast}(T))=\inv(\tau(T))\ ,
\end{align*}
we present the following five technical lemmas
which follow from direct computations.

\begin{lemma}\label{lemma:casexAx}
Let $T=\young(ab,A\ )$ and $T'=\young(ba,A\ )\ .$
If  $T$  satisfies the condition $xAx$,
then 
\begin{equation*}
\numof{\Inv_{2}(T)}=\inv_{2,1}(T') \ .
\end{equation*}
\end{lemma}
\begin{lemma}\label{lemma:casenotxAx}
If a filling $T=\young(ab,A\ )$ does not satisfy the condition $xAx$,
then
\begin{equation*} 
\numof{\Inv_{2}(T)}=\inv_{2,1}(T)\ .
\end{equation*}
\end{lemma}
\begin{lemma}\label{lemma:stabI}
Let $T=\young(ab,AB)$ and $T'=\young(ab,BA)$ be two fillings such that $T$ satisfies one of the following conditions
\begin{equation*}
\begin{array}{ccc}
a,b \leq A,B\ , & \quad & a \leq A,B <b \ , \\
A,B < a,b \ ,   & \quad & b \leq A,B <a\ .
\end{array}
\end{equation*}
Hence, we have the following relations
\begin{align*}
\Des_{2,1}(T) & =  \Des_{2,1}(T')\ , \\
\Inv_{2,1}(T) & =  \Inv_{2,1}(T')\ , \\
\Inv_{2}(T)   & =  \Inv_{2}(T')\ .
\end{align*}
\end{lemma}
\begin{lemma}\label{lemma:stabII}
Let $T=\young(ab,AB)$ and $T'=\young(ab,BA)$ be two fillings such that $T$ satisfies one of the following conditions
\begin{gather*}
A < a,b \leq B,\\
B < a,b \leq A.
\end{gather*}
Hence, we have
\begin{equation*}
\numof{\Des_{2,1}(T)}= \numof{\Des_{2,1}(T')} 
\quad \text{and} \quad 
\inv_{2,1}(T)= \inv_{2,1}(T')\ .
\end{equation*}
\end{lemma}
\begin{lemma}\label{lemma:stabxXxX}
Let $T=\young(ab,AB)$ and $T'=\young(ba,BA)$ be two fillings such that 
$T$ satisfies the condition $xXxX$.
Hence, we have
\begin{equation*}
\numof{\Des_{2,1}(T)}= \numof{\Des_{2,1}(T')} 
\quad \text{and} \quad
\inv_{2,1}(T)= \inv_{2,1}(T')\ .
\end{equation*}
\end{lemma}
\begin{lemma}\label{lemma:AdB}
Let $T=\young(ab,AB)$ be a filling which satisfies $A\neq B$.
Then, $T$ satisfies the condition $xXxX$ or the conditions used in
 Lemma \ref{lemma:stabI} and \ref{lemma:stabII}. 
\end{lemma}
Lemmas \ref{lemma:stabI}, \ref{lemma:stabII} and \ref{lemma:stabxXxX}
imply the following key lemma.

\begin{lemma}\label{lemma:casexXxX}
In Algorithm \ref{def:tau}, the swapping of the value in the $i+1$-th row
 when the $i$-th and the $i+1$-th rows are in the condition $xXxX$ does
 not change the statistic 
$\maj_{i+1,i}$ and $\inv_{i+1,i}$\ .
\end{lemma}
\begin{proof}
If the $i$-th and $(i+1)$-th row satisfy the condition $xXxX$,
then the values of the $(i+1)$-th row are different from each other. 
Using Lemma \ref{lemma:AdB}, we obtain that
the $i$-th and $(i+1)$-th row satisfy the condition $xXxX$
or the conditions of Lemma \ref{lemma:stabI} and \ref{lemma:stabII}.
Hence, it follows from
 Lemmas \ref{lemma:stabI}, \ref{lemma:stabII} and \ref{lemma:stabxXxX}
that 
$$\inv_{i+1,i}(T)=\inv_{i+1,i}(\tau(T))\ .$$
The lemmas also imply
$\numof{\Des_{i+1,i}(T)}=\numof{\Des_{i+1,i}(\tau(T))}$.
In this case, 
$$\maj_{i+1,i}(T) =(k+l-i)\numof{\Des_{i+1,i}(T)} \quad 
\text{and} \quad \maj_{i+1,i}(\tau(T)) =(k+l-i)\numof{\Des_{i+1,i}(\tau(T))}\ .$$
Finally, \vspace{-0.2cm} 
$$\maj_{i+1,i}(T)=\maj_{i+1,i}(\tau (T))\ .$$
\end{proof}

Now we can finish the proof of Theorem \ref{ThDM}. 
Lemmas \ref{lemma:casexAx}, \ref{lemma:casenotxAx}
and \ref{lemma:casexXxX} 
imply the second statement of the theorem
$$\inv(\pi^{\ast}(T))=\inv(\tau(T))\ .$$
It also follows from these lemmas that 
$$\maj(\pi^{\ast}(T))+l\cdot\numof{\Des_{k+1,k}(T)}=\maj(\tau(T))\ .$$
which implies the first statement on statistic $maj$
$$\maj(\pi^{\ast}(T))\equiv(\tau(T)) \pmod{l}\ .$$
\begin{remark}
We can mention that the $(q,t)$-Kostka polynomials
 $K_{\lambda,\mu}(q,t)$ (coefficient of the expansion of the modified
 Macdonald polynomials on the Schur basis) for the special case of two
 columns partitions $\mu=(2^r1^{n-2r})$ have been studied in \cite{Stem}
 and combinatorially interpreted with rigged configurations  in
 \cite{Fishel}. An other approach using statistics on Young tableaux has been developed in \cite{Z} and \cite{LM}.
\end{remark}


\begin{thebibliography}{1}

\bibitem[DM]{DM} F. Descouens and H. Morita, {\it Factorization formula for Macdonald polynomials}, European Journal of Combinatorics (to appear).

\bibitem[F]{Fishel} S. Fishel, {\it Statistics for Special $q.t$-Kostka polynomials}, Proc. Amer. Math. Soc. Vol. {\bf 123}, No. 10, (1995), pp. 2961-2969. 

\bibitem[HHL]{HHL} J. Haglund, M. Haiman and N. Loehr, {\it A combinatorial formula for Macdonald polynomials}, J. Amer. Math. Soc. {\bf 18}, (2005), pp. 195-232.

\bibitem[LM]{LM} L. Lapointe, J. Morse, {\it Tableaux statistics for two parts Macdonald polynomials}, preprint math.CO/9812001.

\bibitem[M1]{Macdo1}I.G. Macdonald, {\it A new class of symmetric
	    functions}, 
Actes du 20e S\'eminaire Lothatingien de Combinatoire, vol. {\bf 372/S-20}, Publications I.R.M.A., Strasbourg, (1998), pp. 131-171.

\bibitem[M2]{Macdo2} I.G. Macdonald, {\it Symmetric functions and Hall polynomials}, second ed. The Clarendon Press, Oxford University Press, New-York, 1995.

\bibitem[S]{Stem} 
J.R. Stembridge, {\it Some particular entries of the two parameter Kostka matrix}, Proc. Amer. Math. Soc., {\bf 121}, (1994), 367-373.

\bibitem[Z]{Z}
M. Zabrocki, {\it A Macdonald Vextex Operator and Tableaux Statistics for the Two-Column $(q,t)$-Kostka Coefficients}, Electronic Journal of Combinatorics, {\bf 5}, R45, (1998), 46p.

\end{thebibliography}
\end{document}